\documentclass[12pt,a4paper]{amsart}
\usepackage{amsmath,amsfonts,amssymb,amsthm,amscd}
\usepackage{dsfont,mathrsfs,enumerate}
\usepackage{curves,epic,eepic}

\usepackage[colorlinks=true,
linkcolor=red, citecolor=blue, pagebackref]{hyperref}

\usepackage{url}
\usepackage{cancel}

\usepackage[pdftex]{graphicx}
\usepackage{psfrag}
\usepackage{epic}
\usepackage{eepic}
\usepackage[matrix,arrow,curve]{xy}
\input{epsf}
\usepackage{tikz}
\usetikzlibrary{calc}
\usetikzlibrary{intersections}
\usetikzlibrary{through}
\usetikzlibrary{backgrounds}

\newtheorem{thm}{Theorem}[section]
\newtheorem{lemma}[thm]{Lemma}
\newtheorem{prop}[thm]{Proposition}
\newtheorem{coro}[thm]{Corollary}

\theoremstyle{definition}
\newtheorem{defi}[thm]{Definition}
\newtheorem{rem}[thm]{Remark}
\newtheorem{ex}[thm]{Example}

\def\Z{\mathds Z}
\def\Q{\mathds Q}
\def\R{\mathds R}

\def\phi{\varphi}
\def\<{{\langle}}
\def\>{{\rangle}}

\newcommand{\lw}[1]{\text{lw}(#1)}

\newcommand{\interior}[1]{\text{int}(#1)}
\newcommand{\floor}[1]{\left\lfloor #1 \right\rfloor}
\newcommand{\ceil}[1]{\left\lceil #1 \right\rceil}

\newcommand{\Min}[2]{\mathrm{Min}_{#1}(#2)}
\newcommand{\conv}[1]{\mathrm{conv}\left(#1\right)}
\newcommand{\normalfan}[1]{\Sigma_{#1}}

\definecolor{zzttqq}{rgb}{0.6,0.2,0}
\definecolor{ududff}{rgb}{0.3,0.3,1}
\definecolor{aabbcc}{rgb}{1,0,0}

\begin{document}
	
	\title[The Fine interior of dilations of a rational polytope]{The Fine interior of dilations of a rational polytope}
	
	\author[Martin Bohnert]{Martin Bohnert}
	\address{Mathematisches Institut, Universit\"at T\"ubingen,
		Auf der Morgenstelle 10, 72076 T\"ubingen, Germany}
	\email{martin.bohnert@uni-tuebingen.de}
	
	\begin{abstract}
	A nondegenerate toric hypersurface of negative Kodaira dimension can be characterized by the empty Fine interior of its Newton polytope according to recent work by Victor Batyrev, where the Fine interior is the rational subpolytope consisting of all points which have an integral distance of at least $1$ to all integral supporting hyperplanes of the Newton polytope. Moreover, we get more information in this situation if we can describe how the Fine interior behaves for dilations of the Newton polytope, e.g. if we can determine the smallest dilation with a non-empty Fine interior. Therefore, in this article we give a purely combinatorial description of the Fine interiors of all dilations of a rational polytope, which allows us in particular to compute this smallest dilation and to classify all lattice $3$-polytopes with empty Fine interior, for which we have only one point as Fine interior of the smallest dilation with non-empty Fine interior.
	\end{abstract}
	
	\maketitle
	
	\thispagestyle{empty}
	
	\section{Introduction}
	
	The Fine interior $F(P)$ of a lattice $d$-polytope $P\subseteq \R^d$ was introduced by Jonathan Fine in \cite[§4.2.]{Fin83} to study top degree differentials of hypersurfaces using their Newton polytope. Originally called \textit{heart} of the Newton polytope, we now use the term \textit{Fine interior} as introduced by Miles Reid when working with plurigenera and canonical models of nondegenerate toric hypersurfaces in \cite[II, 4 Appendix]{Rei87}.
	
	If the Fine interior of the Newton polytope of a nondegenerate hypersurface in a torus is not empty, then we can construct both a unique projective model with at worst canonical singularities and minimal models of the hypersurface using the Fine interior and related combinatorial data, as recently shown by Victor Batyrev in \cite{Bat23}. However, if the Fine interior is empty, the hypersurface has negative Kodaira dimension, and we either get the Zariski closure of the hypersurface in a suitable canonical toric $\Q-$Fano variety as a canonical $\Q$-Fano hypersurface, or we get a $\Q$-Fano fibration for the compactification of the hypersurface by \cite{Bat24}.
	
	We focus here on the second case with an empty Fine interior and call a lattice polytope with an empty Fine interior \textit{$F$-hollow}. To study this case in detail, it is necessary to understand how the Fine interior behaves for dilations of the lattice polytope. The most important affine unimodular invariants of the lattice polytope $P$ in this context are the \textit{minimal multiplier} $\mu=\mu(P)$, which is the smallest real number $\lambda$ with $F(\lambda P)\neq \emptyset$, and the associated Fine interior $F(\mu P)$. We call an $F$-hollow lattice polytope $P$ \textit{weakly sporadic}, if $\dim(F(\mu P))=0$, which corresponds to the first situation with a canonical $\Q$-Fano hypersurface. If $\dim(F(\mu P))>0$, we get the $\Q$-Fano fibration with a combinatorial description from the lattice projection of $P$ along $F(\mu P)$.
	
	There are two main goals of this article: first to give a simple purely combinatorial description of the Fine interiors of all dilations of a rational polytope (Theorem \ref{main_theorem}), which allows in particular to compute the minimal multiplier and to see it as a rational number (Corollary \ref{min_mult_rational}), and, as a second goal, to use these tools to classify all weakly sporadic $F$-hollow $3$-polytopes (Theorem \ref{weakly_sporadic classification}, Theorem \ref{sporadic classification}).
	
	We begin by fixing the notation and give precise definitions of the Fine interior and related combinatorial objects in the next section. The third section gives the description of the Fine interior of dilations of a lattice polytope as the intersection of a hyperplane with a polyhedron of dimension $d+1$, before we take a closer look at the commutativity of intersections with hyperplanes and the computation of the Fine interior in section 4. In section 5 we introduce the minimal and some other special multipliers that mark combinatorial changes for the Fine interior during the dilation of the polytope. And in the last section we give the classification of the weakly sporadic $F$-hollow polytopes in dimension $3$.

\section{The Fine interior and associated combinatorial objects} 

Let be $d\in \Z_{\geq 1}$, $M\cong \Z^d$ be a lattice of rank $d$, $N=\mathrm{Hom}(M,\Z)$ be the lattice dual to $M$ and $\left<\cdot,\cdot\right>\colon M\times N \to \Z$ be the natural pairing. We have real vector spaces $M_\R = M \otimes \R$ and $N_\R = N \otimes \R$ and extend the natural pairing to them.

We will look now at convex geometric objects in $M_\R$ and $N_\R$ and give a brief description of the objects we will need later. For more details see e.g. \cite{Zie07}. The convex hull of a finite number of points in $M_\Q\subseteq M_\R$ is called a \textit{rational $d$-polytope} if there are $d+1$ affine independent points among these points. A \textit{lattice $d$-polytope} is a rational $d$-polytope with all vertices in $M$. For $\lambda \in \R$ and a rational $d$-polytope $P\subseteq M_\R$ we define the real dilation
\begin{align*}
\lambda P := \{\lambda\cdot x \in M_\R \mid x \in P\} \subseteq M_\R.
\end{align*}
By a \textit{$d$-polytop} in $M_\R$ we mean in the following a real dilation of some rational $d$-polytope, and these are the most general polytopes we work with.

If a dual vector $y \in N_\R\setminus \{0\}$ attains its minimum on a set $P\subseteq M_\R$, we denote this minimum by
\begin{align*}
	\Min{P}{y}:= \min \{\left<x,y\right> \mid x\in P\}
\end{align*}
and we write $\Min{P}{y}=-\infty$ if for all $n\in \Z_{\geq 0}$ there is a point $p_n\in P$ with $\left<p_n,y\right><-n$. We can use this notation, for example, to describe a convex, closed set $P\subseteq M_\R$ with help of the supporting hyperplane theorem as an intersection of closed half-spaces by 
\begin{align*}
	P =& \{x\in M_\R \mid \left<x,y\right> \geq \Min{P}{y} \ \forall y \in N_\R\setminus \{0\} \}.
\end{align*}
For a $d$-polytope $P\subseteq M_\R$ we define additional data to simplify this dual description. For a face $F \preceq P$, i.e. the intersection of $P$ with a supporting hyperplane, we have a cone
\begin{align*}
	\sigma_F := \{ y \in N_\R \mid \Min{P}{y}=\Min{F}{y} \} \subseteq N_\R,
\end{align*}
the \textit{normal cone of the face $F$}. All normal cones together form a fan, which we call the \textit{normal fan $\normalfan{P}$} of the polytope $P$. We write $\Sigma_P[1]\subseteq N$ for the set of primitive ray generators of the normal fan, i.e. the primitive dual lattice vectors that generate the cones of dimension $1$. We can now use this to uniquely describe our $d$-polytope as the intersection of the finitely many closed half-spaces associated to facets, i.e. faces of dimension $d-1$, by
\begin{align*}
	P = \{x\in M_\R \mid \left<x,n\right> \geq \Min{P}{n} \ \forall n \in \Sigma_P[1] \}.
\end{align*}
More generally, we call the possibly unbounded intersection of finitely many closed half-spaces in $M_\R$ a \textit{polyhedron}. A \textit{rational polyhedron} is a polyhedron, where we can chose for each facet of the polyhedron a normal vector from $N$. If we have $0$ in each facet, we have a \textit{polyhederal cone}. We can associate to each polytope in $M_\R$ the \textit{cone over $P$}, i.e. the polyhedron cone defined by
\begin{align*}
	\mathrm{cone}(P):= \R_{\geq 0} (\{1\} \times P)\subseteq \R \times M_R.
\end{align*}
Moreover, every polyhedron has a decomposition as a Minkowski sum of a polytope and a polyhedral cone where the polyhedral cone is uniquely determined and is called the \textit{recession cone of the polyhedron}.

We will now introduce the Fine interior of a $d$-polytope as the central combinatorial object of the article.

\begin{defi}\label{Def_FineInt}
Let $P\subseteq M_\R$ be a $d$-polytope. Then the \textit{Fine interior $F(P)$} of $P$ is defined as
\begin{align*}
F(P) := \{x\in M_\R \mid \left<x,n\right> \geq \Min{P}{n}+1 \ \forall n \in N\setminus \{0\} \} \subseteq P.
\end{align*}
\end{defi}

In fact, we do not need to use all the dual vectors of $N\setminus \{0\}$ in the definition \ref{Def_FineInt}. It is enough to choose the finitely many dual vectors that are part of a Hilbert basis of some cone of the normal fan $\Sigma_P$ (\cite[3.10]{Bat23}). In particular, since $F(P)$ is bounded as a subset of $P$, the Fine interior of a $d$-polytope is itself a $d$-polytope. By \cite[5.4]{Bat23} it is even possible to restrict to primitive ray generators in the canonical refinement of $\Sigma_P$, i.e. elements of
\begin{align*}
\normalfan{P}^{\mathrm{can}}[1] = \{ n \in N \mid n \text{ is vertex of } \conv{\sigma \cap (N\setminus \{0\})} \text{ for a maximal cone } \sigma \in \Sigma_P\}
\end{align*}
and so we get
\begin{align*}
	F(P) = \{x\in M_\R \mid \left<x,n\right> \geq \Min{P}{n}+1 \ \forall n \in \normalfan{P}^{\mathrm{can}}[1] \}.
\end{align*}
But this is the strongest possible a priori restriction on dual lattice vectors, since by \cite[5.3]{Bat23}  we need  all $n \in \normalfan{P}^{\mathrm{can}}[1]$ if we look at $\lambda P$ for $\lambda$ large enough instead of $P$.

A posteriori, we could use in the definition \ref{Def_FineInt}  only those dual vectors that define half-spaces which really touch $F(P)$ after the shift by $1$. We have the following formal definition for these dual vectors:

\begin{defi}
Let $P\subseteq M_\R$ be a $d$-polytope with $F(P)\neq \emptyset$. Then the \textit{support of $F(P)$} is the set of dual lattice vectors
\begin{align*}
	S_F(P):=\{n \in N \mid \Min{F(P)}{n} = \Min{P}{n}+1\}\subseteq N.
\end{align*}
\end{defi}

In the construction of canonical models of nondegenerate hypersurfaces \cite{Bat23} there is another rational polytope defined by the support of the Fine interior which turns out to be very helpful, which we introduce now.

\begin{defi}
	Let $P\subseteq M_\R$ be a $d$-polytope with $F(P)\neq \emptyset$. Then the \textit{canonical hull of $P$} is the $d$-polytope
	\begin{align*}
		C(P):=\{x\in M_\R \mid \left<x,\nu\right> \geq \Min{P}{\nu} \ \forall \nu \in S_F(P) \}\subseteq P.
	\end{align*}
	Moreover, $P$ is called \textit{canonically closed} if $C(P)=P$, i.e. if and only if we have $\Sigma_{P}[1]\subseteq S_F(P)$.
\end{defi}

We end this section with two important examples of canonically closed lattice polytopes that we will need later.

\begin{ex}\label{Ex_canonically_closed_dim2}
All lattice polygons, i.e. lattice polytopes of dimension $d=2$, with non-empty Fine interior are canonically closed by \cite[4.4]{Bat23}.
\end{ex}

\begin{ex}\label{Ex_canonically_closed_reflexive}
In arbitrary dimension the \textit{reflexive polytopes}, i.e. the lattice polygons $P\subseteq M_\R$, which have $0\in \interior{P}$ and for which the dual polytope 
\begin{align*}
	P^* := \{ y \in N_\R \mid \left<x,y\right> \geq -1 \ \forall x \in P\} \subseteq N_\R
\end{align*}
is a lattice polytope, are canonically closed. More precisely, they are exactly those canonically closed lattice polytopes which have $\{0\}$ as their Fine interior (\cite[4.9]{Bat23}). Reflexive polytopes are well known because of their important role in combinatorial mirror symmetry (\cite{Bat94}). Note that we also get interesting objects for combinatorial mirror symmetry in some cases when the lattice polytope has only $\{0\}$ as its Fine interior, but it is not canonically closed (\cite{Bat17}).
\end{ex}

\section{Simultaneous description of the Fine interiors of all dilations}
	
In this section we want to describe the Fine interiors of all real dilations of a $d$-polytope simultaneously. As a first step, we want to understand the behavior of $\Min{P}{y}$, $\normalfan{P}$ and $S_F(P)$ under real dilations. We do this in the following two lemmata.

\begin{lemma}\label{Dilations_and_normalfans}
Let $P\subseteq M_\R$ be a $d$-polytope and $\lambda \in \R_{> 0}$. Then we have for all $y\in N_\R$ that $\lambda\Min{P}{y}=\Min{\lambda P}{y}$ and $\Sigma_P= \Sigma_{\lambda P}$.
\end{lemma}
\begin{proof}
Since $\lambda > 0$ the multiplication with $\lambda$ respects the minimum and due to the linearity of $\left<\cdot, y\right>$ we have
\begin{align*}
	\lambda\Min{P}{y}=& \lambda \cdot \min \{\left<x,y\right> \mid x\in P\}= \min \{\left<\lambda x,y\right> \mid x\in P\} = \min \{\left<x,y\right> \mid x\in \lambda P\} \\=& \Min{\lambda P}{y}.
\end{align*}
If we use this not only for $P$, but also for a face $F\preceq P$, we get for the normal cones
\begin{align*}
	\sigma_F = \{ y \in N_\R \mid \Min{P}{y}=\Min{F}{y} \} = \{ y \in N_\R \mid \Min{\lambda P}{y}=\Min{\lambda F}{y} \}= \sigma_{\lambda F}
\end{align*}
and so we have $\Sigma_P=\Sigma_{\lambda P}$.
\end{proof}

\begin{lemma}\label{Dilations_and_support}
Let $P\subseteq M_\R$ be a $d$-polytope with $F(P)\neq \emptyset$. If $\nu \in S_F(P)$, then $\nu \in S_F(\lambda P)$ for all $\lambda \in \R_{\geq 1}$. If $F(P)$ is full-dimensional and $\nu \in \Sigma_{F(P)}[1]$, then we also have $\nu \in \Sigma_{F(\lambda P)}[1]$ for all $\lambda \in \R_{\geq 1}$.
\end{lemma}
\begin{proof}
Since $\nu \in S_F(P)$, there is a point $x\in F(P)$ with $\left<x,\nu\right>=\Min{P}{\nu}+1$ and
$\left<x,n\right>\geq \Min{P}{n}+1$ for all $n\in N\setminus \{0\}$. For each $n\in N\setminus \{0\}$ we take some $x_n\in P$ with $\left<x_n,n\right>=\Min{P}{n}$ and set $w_n:=x-x_n\in M_\R$. Then we have $x=x_\nu+w_\nu=x_n+w_n$ for all $n\in N\setminus\{0\}$ and 
\begin{align*}
\left<w_\nu,\nu\right>=\left<x-x_\nu,\nu\right>=1, \left<w_n,n\right>=\left<x-x_n,n\right>\geq 1.
\end{align*}
Thus we get for $\lambda \in \R_{\geq 1}$ and $n\in N\setminus\{0\}$ with \ref{Dilations_and_normalfans}
\begin{align*}
\left<\lambda x_\nu+w_\nu,n\right>=&\left<\lambda(x_n+w_n-w_\nu)+w_\nu,n\right>\\=& \lambda \left<x_n,n\right>+\lambda\left<w_n,n\right>-(\lambda-1)\left<w_\nu,n\right>\\= &\lambda \Min{P}{n}+\lambda\left<w_n,n\right>-(\lambda-1)\left<x_n+w_n-x_\nu,n\right>\\
= &\Min{\lambda P}{n}+\left<w_n,n\right>-(\lambda-1)\Min{P}{n}+(\lambda-1)\left<x_\nu,n\right>\\
\geq & \Min{\lambda P}{n}+1,
\end{align*}
because $x_\nu \in P$ and so $\left<x_\nu,n\right>\geq \Min{P}{n} $.
So $\lambda x_\nu+w_\nu \in F(\lambda P)$ and we get $\nu \in S_F(\lambda P)$ by
\begin{align*}
	\left<\lambda x_\nu+w_\nu,\nu\right>=\lambda \Min{P}{\nu}+1=\Min{\lambda P}{\nu}+1.
\end{align*}
If $\nu \in \Sigma_{F(P)}[1]$, then we can do the same calculations not only for $x$ but also for $d$ affine independent points on the corresponding facet to $\nu$ of $F(P)$ and get $d$ affine independent points on a facet of $F(\lambda P)$ defined by $\nu$ and thus $\nu \in \Sigma_{F(\lambda P)}[1]$ for all $\lambda \in \R_{\geq 1}$.
\end{proof}

We are now ready to describe the Fine interior of all real dilations of a $d$-polytope in $M_\R$ simultaneously with the help of a suitable polyhedron in $\R \times M_\R$, which lies in the cone over the polytope.
	
\begin{thm}\label{main_theorem}
Let $P\subseteq M_\R$ be a $d$-polytope and $\lambda \in \R_{\geq 0}$. Then
\begin{align*}
\mathcal{F}(P):=&\left\{(x_0,x)\in \R \times M_\R \mid \left<(x_0,x),(-\Min{P}{\nu},\nu)\right> \geq 1 \ \forall \nu \in \normalfan{P}^{\mathrm{can}}[1]  \right\} \subseteq \R \times M_\R
\end{align*}
is a polyhedron without compact facets, the number of facets is $|\normalfan{P}^{\mathrm{can}}[1]|$, the recession cone of $\mathcal{F}(P)$ is $\mathrm{cone}(P)=\R_{\geq 0}(\{1\} \times P)$, and for all $\lambda \in \R_{\geq 0}$ we get  the Fine interior of $\lambda P$ from $\mathcal{F}(P)$ by the identity
\begin{align*}
\{\lambda \} \times F(\lambda P) = \mathcal{F}(P)\cap \{(x_0,x)\in \R \times M_\R \mid x_0=\lambda\}.
\end{align*}
\end{thm}
\begin{proof}
The Fine interior of $\lambda P$ is given by
\begin{align*}
F(\lambda P)=\{x\in M_\R \mid \left<x,\nu\right> \geq \Min{\lambda P}{\nu}+1 \ \forall \nu \in \normalfan{\lambda P}^\mathrm{can}[1]\}.
\end{align*}
With \ref{Dilations_and_normalfans} we get
\begin{align*}
F(\lambda P)=&\{x\in M_\R \mid \left<x,\nu\right> \geq \lambda \Min{P}{\nu}+1 \ \forall \nu \in \normalfan{P}^\mathrm{can}[1]\}\\
			=&\{x\in M_\R \mid \left<(\lambda,x),(-\Min{P}{\nu},\nu)\right> \geq 1 \ \forall \nu \in \normalfan{P}^\mathrm{can}[1]  \}
\end{align*}
and so we have
\begin{align*}
&\{\lambda\} \times F(\lambda P)\\=&\left\{(x_0,x)\in \R \times M_\R \mid \left<(x_0,x),(-\Min{P}{\nu},\nu)\right> \geq 1 \ \forall \nu \in \normalfan{P}^\mathrm{can}[1]  \right\}\cap \{x_0=\lambda\}.
\end{align*}
Since $\normalfan{P}[1]\subseteq \normalfan{P}^{\mathrm{can}}[1]$  we get the recession cone of $\mathcal{F}(P)$ by \cite[Prop. 1.12.]{Zie07} as
\begin{align*}
&\left\{(x_0,x)\in \R \times M_\R \mid \left<(x_0,x),(-\Min{P}{\nu},\nu)\right> \geq 0 \ \forall \nu \in \normalfan{P}^{\mathrm{can}}[1]  \right\}\\
=& \left\{(x_0,x)\in \R \times M_\R \mid \left<x,\nu\right> \geq x_0\cdot \Min{P}{\nu} \ \forall \nu \in \normalfan{P}^{\mathrm{can}}[1] \right\}\\
=& \left\{(x_0,x)\in \R \times M_\R \mid x\in x_0\cdot P \right\}\\
=& \mathrm{cone}(P).
\end{align*}
The facets of $\mathcal{F}(P)$ are unbounded because from $\nu \in \normalfan{F(\lambda_0P)}[1]$ we get $\nu \in \normalfan{F(\lambda P)}[1]$ for all $\lambda\geq \lambda_0$ by \ref{Dilations_and_support} and the number of facets is $|\normalfan{P}^{\mathrm{can}}[1]|$ since we need all vertices of the canonical refinement to describe $F(\lambda P)$ for $\lambda$ large enough.		
\end{proof}

We now focus on the situation where $P$ is at least a rational polytope, so that $\mathcal{F}(P)$ also becomes a rational polyhedron. For lattice polytopes we want to give an alternative description of $\mathcal{F}(P)$ using the Fine interior of $\mathrm{cone}(P)$. Although  we have not yet given a formal definition of the Fine interior of cones, it seems reasonable to have this definition completely analogous. Since $n \in N \setminus \{0\}$ attains a finite minimum on $\sigma$ if and only if $n$ is an element of the dual cone
\begin{align*}
	\check{\sigma}=\{y\in N_\R \mid \left<x,y\right> \geq 0 \ \forall x \in \sigma \} \subseteq N_\R
\end{align*}
and then the minimum is $0$, we have
\begin{align*}
	F(\sigma) :=& \{x\in M_\R \mid \left<x,n\right> \geq \Min{\sigma}{n}+1 \ \forall n \in N\setminus \{0\} \}\\
	=& \{x\in M_\R \mid \left<x,n\right> \geq 1 \ \forall n \in \check{\sigma} \cap N\setminus \{0\} \}.
\end{align*}
This allows us in some cases to get $\mathcal{F}(P)$ directly as the Fine interior of $\mathrm{cone}(P)$, as the following corollary shows.

\begin{coro}\label{Fine_interior_cone}
	If $P$ is a lattice polytope, then we have for all $\lambda\in \R_{\geq 1}$ that
	\begin{align*}
		\mathcal{F}(P)\cap \{x_0=\lambda\} =& F(\lambda P) \times \{\lambda\} = F(\mathrm{cone}(P)\cap \{x_0=\lambda\}) \times \{\lambda\}\\ =&F(\mathrm{cone}(P)) \cap \{x_0=\lambda\}.
	\end{align*}
	In particular, we have $\mathcal{F}(P)=F(\mathrm{cone}(P))$ if and only if $F(\lambda P)=\emptyset$ for all $\lambda<1$.
\end{coro}
\begin{proof}
	By \ref{main_theorem} and the definition of $\mathrm{cone}(P)$  we get
	\begin{align*}
		\mathcal{F}(P)\cap \{x_0=\lambda\}=F(\lambda P) \times \{\lambda\} =F(\mathrm{cone}(P)\cap \{x_0=\lambda\}) \times \{\lambda\}.
	\end{align*}
	Since we have by definition
	\begin{align*}
	&F(\mathrm{cone}(P))=\\ &\{(x_0,x)\in \R \times M_\R \mid \left<(x_0,x),(n_0,n)\right> \geq 1 \ \forall (n_0,n) \in \check{\mathrm{cone}(P)} \cap (\Z \times N) \setminus \{(0,0)\} \},
	\end{align*}
	it is by \ref{main_theorem} and for $\lambda\geq 1$ enough to see that the vertices of
	\begin{align*}
	\conv{\check{\mathrm{cone}(P)} \cap (\Z \times N\setminus \{0\} )}
	\end{align*}
	are exactly given by $(-\Min{P}{\nu},\nu)$ with $\nu \in \normalfan{P}^\mathrm{can}$. First notice that $-\Min{P}{\nu}\in \Z$ since $P$ is a lattice polytope. Also, for all $x\in P$ we have 
	\begin{align*}
		\left<(-\Min{P}{\nu},\nu),(\lambda,\lambda x)\right>=-\lambda \Min{P}{\nu}+\lambda\left<\nu,x\right>)\geq 0
	\end{align*}
	and therefore $(-\Min{P}{\nu},\nu)\in \check{\mathrm{cone}(P)}$. With the same calculation we see that the facet $F_p\preceq \check{\mathrm{cone}(P)}$ corresponding to a vertex $p\preceq P$ is
	\begin{align*}
		F_p=\{(y_0,y)\in (\Z \times N)_\R \mid y \in \sigma_p, y_0=-\Min{P}{y} \}.
	\end{align*}
	Thus every lattice point $(n_0,n)\in \check{\mathrm{cone}(P)}\cap (\Z \times N\setminus \{0\})$ shares its $N$-projection with a lattice point in a facet, namely $(-\Min{P}{n},n)$. Thus the convex hull of $\check{\mathrm{cone}(P)}\cap (\Z \times N\setminus \{0\})$ is the convex hull of the lattice points different from $0$ in the facets, and the vertices of this convex hull are the points $(-\Min{P}{\nu}, \nu)$ with $\nu \in \normalfan{P}^\mathrm{can}$ by definition of the canonical refinement.	
\end{proof}

We end this section with a characterization of the cases where $\mathcal{F}(P)$ is as simple as possible, i.e. it is $\mathcal{F}(P)$ is a rational cone shifted by a lattice vector. For this, recall the following notions for Fano polytopes (for more background on Fano polytopes see e.g. the survey \cite{KN12}).

\begin{defi}
Let $P\subseteq M_\R$ be a lattice $d$-polytope. If $\mathrm{int}(P)\cap M=\{0\}$, then we call $P$ a \textit{canonical Fano polytope}. If there is some $k\in \Z_{\geq 1}$ such that $kP$ is affine unimodular equivalent to a reflexive polytope, then we call $P$ a \textit{Gorenstein polytope of index $k$}.
\end{defi}

We now get the following characterizations for the cases where $\mathcal{F}(P)$ is a rational cone shifted by a lattice vector.

\begin{coro}\label{Fine_lattice_cone}
	Let $P$ be a rational polytope. Then $\mathcal{F}(P)$ is a rational cone shifted by $(1,0)$ if and only if $P^*$ is a canonical Fano polytope.
\end{coro}
\begin{proof}
	Since $\mathrm{cone}(P)$ is the recession cone of $\mathcal{F}(P)$ by \ref{main_theorem}, $\mathcal{F}(P)$ is a rational cone shifted by $(1,0)$ if and only if $\mathcal{F}(P)=(1,0)+\mathrm{cone}(P)$. But by \ref{main_theorem} this is equivalent to $\Sigma_P[1]=\Sigma_P^{\mathrm{can}}[1]$, $F(P)=\{0\}$, and $\Min{P}{\nu}=-1$ for all $\nu \in \Sigma_P[1]$, which we have if and only if $P^\ast$ is a lattice polytope having only $0$ as an interior lattice point.
\end{proof}

\begin{coro}
	Let $P$ be a lattice polytope. Then $\mathcal{F}(P)$ is a lattice cone with vertex $(k,x), k\in \mathbb{Z}_{\geq 1}, x\in M$ if and only if $kP-x$ is a reflexive polytope. In particular, $P$ is then a Gorenstein polytope of index $k$.
\end{coro}

\begin{proof}
	$\mathcal{F}(P)$ is a lattice cone with vertex $(k,x)\in \Z \times M$ if and only if $\mathcal{F}(kP-x)$ is a lattice cone with vertex $(0,1)$. This is equivalent by \ref{Fine_lattice_cone} to the situation where $(kP-x)^*$ is a canonical Fano polytope, and since $kP-x$ is a lattice polytope, this is the case if and only if $kP-x$ is reflexive.
\end{proof}

\section{The Fine interior and intersections with hyperplanes}

The result in \ref{Fine_interior_cone} motivates us to look at cases where we can do Fine interior computations in codimension $1$ using intersections with hyperplanes. We will see in this section that we have such results for dilations of lattice polytopes with small lattice width. Recall that the \textit{lattice width} $\lw{P}$ of a lattice polytope $P\subseteq M_\R$ is defined by
\begin{align*}
	\lw{P} := \min \{ -\Min{P}{-n}-\Min{P}{n} \mid n \in N\setminus \{0\} \}
\end{align*}
and a dual vector $n_{lw}\in N$ with
\begin{align*}
	\lw{P} = -\Min{P}{-n_{lw}}-\Min{P}{n_{lw}}
\end{align*}
is called a \textit{lattice width direction}.

We start with the lattice pyramid $\mathrm{Pyr}(P)$ of a lattice $d$-polytope $P\subseteq M_\R$, which is defined by
\begin{align*}
	\mathrm{Pyr}(P):=\conv{\{1\}\times P, (0,0)}.
\end{align*}
and clearly has $\lw{\mathrm{Pyr}(P)}=1$. Since we can also consider $\mathrm{Pyr}(P)$ as the intersection of $\mathrm{cone}(P)$ with the half-space $x_0\leq 1$, we get the following corollary from \ref{Fine_interior_cone}.

\begin{coro}\label{Fine_interior_pyramid}
Let $P\subseteq M_\R$ be a lattice $d$-polytope, $\mu,\lambda\in \Q_{\geq 0}$. Then we have for dilations of the lattice pyramid $\mathrm{Pyr}(P)\subseteq \R \times M_\R$
\begin{align*}
F(\mu\mathrm{Pyr}(P))\cap \{x_0=\lambda\}\cong \begin{cases}
F(\lambda P) &\text{ if } 1\leq \lambda \leq \mu-1\\
\emptyset &\text{ else}.
\end{cases}
\end{align*}
and if $F(\mu\mathrm{Pyr}(P))\neq \emptyset$, then $\mu\geq 2$ and
\begin{align*}
&S_F(\mu\mathrm{Pyr}(P))\\=&\begin{cases}\{(-\Min{P}{\nu},\nu)\in \Z \times N \mid \nu \in S_F(\mu-1)P\}\cup \{(0,\pm 1)\} &\text{ if } F(P)\neq \emptyset\\
\{(-\Min{P}{\nu},\nu)\in \Z \times N \mid \nu \in S_F(\mu-1)P\}\cup \{(0,-1)\} &\text{ if } F(P)=\emptyset.	
\end{cases}
\end{align*}
In particular, $F(2\mathrm{Pyr}(P))=\{1\} \times F(P)$ and for $F(P)\neq \emptyset$ we have that $2\mathrm{Pyr}(P)$ is canonically closed if and only if $P$ is canonically closed.
\end{coro}
\begin{proof}
For $(x_0,x)\in F(\mu \mathrm{Pyr}(P))\subseteq \R \times M_\R$ we have
\begin{align*}
	\left<(x_0,x),(1,0)\right>=x_0\geq \Min{\mu\mathrm{Pyr}(P)}{(1,0)}+1=1
\end{align*}
and 
\begin{align*}
	\left<(x_0,x),(-1,0)\right>=-x_0\geq \Min{\mu\mathrm{Pyr}(P)}{(-1,0)}+1=-\mu+1.
\end{align*}
So we get $F(\mu\mathrm{Pyr}(P))\cap \{x_0=\lambda\}=\emptyset$ for $\lambda<1$ and $\lambda \geq \mu-1$ and support vectors $(0,\pm 1)$ in the appropriate cases. If we can show that all the other support vectors of $F(\mu\mathrm{Pyr}(P))$ attain their minimum on $\mu \mathrm{Pyr}(P)$ on a face of dimension $1$ or greater, we get the result as a corollary of \ref{Fine_interior_cone}. If $\mu \in \Z$, this is clear, because $\mu\mathrm{Pyr}(P)\cap \{x_0=1\}\cong P$ and $\mu\mathrm{Pyr}(P)\cap \{y=\mu-1\}\cong(\mu-1)P$ are lattice polytopes in this case, and so there is no addition support vector which attains its minimum only on a vertex of $\mu\mathrm{Pyr}(P)$. For $\mu \in \Q$ we get this, since the situation near vertices after a translation with a suitable vector from $\Q^{d+1}$ is locally the same as for $\floor{\mu}\mathrm{Pyr}(P)$ and the calculation of the Fine interior commutes with translations.
\end{proof}

We can note expect, that we can always commute Fine interior calculations with intersections with hyperplanes. Nevertheless, we have at least the partial result of the following lemma.

\begin{lemma}\label{Intersection_of_Fine_interior}
Let $P\subseteq \R \times M_\R$ be a rational $d$-polytope with $P\cap (\{1\} \times M_\R ) \neq \emptyset$ and $P\cap (\{-1\} \times M_\R ) \neq \emptyset$, and $P_0 \subseteq M_\R$ defined by $\{ 0 \} \times P_0 = P \cap (\{0\} \times M_\R )$. 
Then $\{ 0 \} \times F(P_0) \subseteq F(P) \cap (\{0\} \times M_\R )$.
\end{lemma}
\begin{proof}
If $(0,x)\notin F(P) \cap (\{0\} \times M_\R )$, then we have some $(\nu_0,\nu)\in (\Z \times N)\setminus \{(0,0)\}$ with
\begin{align*}
	\left< (0,x),(\nu_0,\nu)\right> < \Min{P}{(\nu_0,\nu)}+1.
\end{align*}
Note that $\nu \neq 0$, because otherwise we had 
\begin{align*}
	\Min{P}{(\nu_0,0)}+1\leq-|\nu_0|+1\leq 0=\left<(0,x),(\nu_0,0)\right>
\end{align*}
for all $\nu_0\neq 0$, since $P\cap P\cap (\{1\} \times M_\R ) \neq \emptyset\neq P\cap P\cap (\{-1\} \times M_\R )$. So we get $\nu\neq 0$ with
\begin{align*}
	\left<x,\nu\right>=\left< (0,x),(\nu_0,\nu)\right> <  \Min{P}{(\nu_0,\nu)}+1 \leq& \Min{\{0\} \times P_0}{(\nu_0,\nu)}+1=\Min{P_0}{\nu}+1
\end{align*}
and so $x\notin F(P_0)$.
\end{proof}


As the main result of the section we show now, that everything is fine if we have a lattice polytopes of lattice width $2$ whose middle-polytope is also a lattice polytope.

\begin{prop}\label{FineInterior_Width2}
Let $P\subseteq \R \times M_\R$ be lattice polytope of lattice width $2$ with $P\subseteq [-1,1] \times M_\R$, such that $\{0\} \times P_0 :=P\cap (\{0\} \times M_\R)$ is also a lattice polytope.\\
Then
\begin{align*}
F(P)=\{0\} \times F(P_0)
\end{align*}
and if $F(P)\neq \emptyset$, then
\begin{align*}
S_F(P_0)=\{\nu \in N\setminus\{0\} \mid \exists \nu_0\in \Z, (\nu_0,\nu)\in S_F(P)\}
\end{align*}
and $S_F(P)$ is given as union of $\{(-1,0),(1,0)\}$ with
\begin{align*}
\{(\nu_0,\nu) \in \Z \times N \mid \nu\in S_F(P_0), \Min{P}{(\nu_0,\nu)}=\Min{\{0\} \times P_0}{(\nu_0,\nu)}\}.
\end{align*}
\end{prop}
\begin{proof}
Since $P\subseteq [1,1] \times M_\R$ we have $F(P)=F(P)\cap (\{0\} \times M_\R)$ and so we have $F(P)\supseteq \{ 0 \} \times F(P_0)$ by \ref{Intersection_of_Fine_interior}.
	
Now we show $F(P)\subseteq F(P_0)\times \{0\}$. Since $P\subseteq [1,1] \times M_\R$ we have $F(P)\subseteq \{ 0 \} \times P_0$ and it is enough to show that $(0,x)\notin F(P)$ for all $x\notin F(P_0)$. For $x\in M_\R \setminus F(P_0)$ we have some $0\neq \nu \in N$ with
\begin{align*}
\left<x,\nu\right> < \Min{P_0}{\nu}+1.
\end{align*}
Since $P_0$ is a lattice polytope, we have a vertex $e\in M$ of $P_0$ with $\Min{P_0}{\nu}=\left<e,\nu\right>$. Let $f=(1,f')\in M$ be a vertex of $P$ such that $\nu_0:=\left<e-f',\nu\right>\in \Z$ is maximal. Then we have
\begin{align*}
\left<(1,f'),(\nu_0,\nu)\right>=\nu_0+\left<f',\nu\right>=\left<e,\nu\right>=\Min{P_0}{\nu}
\end{align*}
and since $\left<(0,e),(\nu_0,\nu)\right>=\left<e,\nu\right>=\Min{P_0}{\nu}$ and $\nu_0$ was chosen maximal, we get that $(\nu_0,\nu)$ defines a hyperplane supporting $P$ on on a face containing $e$ and $f$ and we have
\begin{align*}
\Min{P}{(\nu_0,\nu)}=\Min{P_0}{\nu}
\end{align*}
and
\begin{align*}
\left<(0,x),(\nu_0,\nu)\right>=\left<x,\nu\right><\Min{P_0}{\nu}+1=\Min{P}{(\nu_0,\nu)}+1,
\end{align*}
i.e. $(0,x)\notin F(P)$.

By the calculations above we get also that for every $\nu \in S_F(P_0)$ there exists a $\nu_0\in \Z$, for example with $\nu_0$ defined as above, with $(\nu_0,\nu)\in S_F(P)$. Since $P_0$ is a lattice polytope and $F(P)\subsetneq P_0$, we have $\Min{P}{(\nu_0,\nu)}=\Min{ \{ 0 \} \times P_0}{(\nu_0,\nu)}$ for all $(\nu_0,\nu)\in S_F(P)\setminus \{(\pm 1,0)\}$, and from $F(P)=\{ 0 \} \times F(P_0)$ we get $\nu \in S_F(P_0)$.
\end{proof}

As a corollary we can use two-dimensional results from \cite{BS24}  to explain some experimental results on the Fine interior of lattice $3$-polytopes in \cite{BKS22}. 

\begin{coro}
	Let  $P \subseteq \R^3$ be a lattice $3$-polytope with $|\interior{P}\cap \Z^3| \leq 1$. Then $\dim F(P)\neq 2$.
\end{coro}
\begin{proof}
	Suppose $\dim F(P) = 2$. Then $P$ must have lattice width $2$ and we can assume $P\subseteq [-1,1] \times \R^2$. We have that the half-integral polygon $P_0\subseteq \R^2$ definied by $\{0\} \times P_0 := P \cap (\{0\} \times \R^2)$ has at least $1$ interior integral point. By classification results of maximal half-integral polygons in \cite[section 5 and 6]{BS24} we know, that then we can assume $P_0 \subseteq \R \times [-1,1]$ or $P_0 \subseteq \conv{(0,0), (3,0), (0,3)}$. In particular, we can assume that $P_0$ is part of a lattice polygon $\bar{P_0}$ which has Fine interior of dimension smaller than $1$. Thus by \ref{FineInterior_Width2} we get $\dim F(\conv{P,\bar{P_0}})<2$ for the lattice polytope $\conv{P,\bar{P_0}}$ and from $F(P) \subseteq F(\conv{P,\bar{P_0}})$ we get the contradiction $\dim F(P)<2$.
\end{proof}

The following example shows that we cannot omit central premises in \ref{FineInterior_Width2}.

\begin{ex}
	For  $P:=\conv{(-1,-1,-1),(1,0,-1),(0,1,-1),(0,0,1)}$, we have
	\begin{align*}
		F(P\cap \{x_3=0\})=\emptyset \neq \{(0,0,0)\}=F(P)\cap \{x_3=0\}=F(P).
	\end{align*}
	Note, that we have $\lw{P}=2$ but
	\begin{align*}
		P\cap \{x_3=0\}=\conv{(-1/2,-1/2),(1/2,0,0),(0,1/2,0)}
	\end{align*}
	is not a lattice polytope. If the intersection gives a lattice polytope, we also cannot commute the calculation of the Fine interior and the intersection in general. See for example
	\begin{align*}
		&F(2P\cap \{x_3=0\})=F((-1,-1),(1,0,0),(0,1,0))=\{(0,0,0)\}\\\neq & \conv{(-1/2,-1/2,0),(1/2,0,0),(0,1/2,0)}=F(2P)\cap \{x_3=0\}.
	\end{align*}
\end{ex}

We will later use the following corollary which helps us to understand the Fine interior for dilations of lattice polytopes of lattice width $1$.

\begin{coro}\label{FineInterior_Width1}
Let P be a lattice polytope of lattice width 1, $P\subseteq [0,1] \times M_\R$ and $\{\frac{1}{2}\} \times P_{1/2}:=P \cap (\{\frac{1}{2}\} \times M_\R)$. Then $2P_{1/2}$ is a lattice polytope with
\begin{align*}
F(2P)=\{1\} \times F(2P_{1/2}).
\end{align*}
Moreover, if $F(2P)\neq \emptyset$, then we get that $2P$ is canonically closed if and only if $2P_{1/2}$ is canonically closed.
\end{coro}
\begin{proof}
From \ref{FineInterior_Width2} we get $F(2P)=\{1\} \times F(2P_{1/2})$ and since $2P$ has no vertices with first coordinate $1$, we have for all $\nu \in S_F(2P_{1/2})$ exactly one $\nu_0 \in \Z$ with $(\nu_0,\nu)\in S_F(2P)$ and $(\nu_0,\nu)$ attains its minimum on $P$ on an edge. So $\Sigma_{P}[1]\in S_F(2P)$ if and only if $\Sigma_{P_{1/2}}[1] \in S_F(2P_{1/2})$.
\end{proof}

We can even go further for $d=2$ since the Fine interior of a lattice polygon is the convex hull of the interior lattice points by \cite[2.9]{Bat17} and every lattice polygon with non-empty Fine interior is canonically closed by \ref{Ex_canonically_closed_dim2}.

\begin{coro}\label{FineInterior_Width1_dim3}
Let $P\subseteq \R \times M_\R$ be a lattice $3$-polytope of lattice width $1$.\\
Then $F(2P)=\conv{\mathrm{int}(2P)\cap M}$ and if $F(2P)\neq \emptyset$, then $2P$ is canonically closed.
\end{coro}

We finish this section with an important example for the situation in the corollary.

\begin{ex}
Let $\Delta\subseteq \R^3$ be an empty lattice $3$-simplex of normalized volume $q$, i.e. $|\Delta \cap \Z^3|=4$. Then we have by \cite{Whi64} some $p\in \Z$ with $gcd(p,q)=1$ and
\begin{align*}
\Delta \cong \Delta(p,q) := \conv{(0,0,0), (1,0,0),(0,0,1),(p,q,1)},
\end{align*}
in particular, the lattice width of $\Delta$ is $1$. By \ref{FineInterior_Width1_dim3} we get
\begin{align*}
F(2\Delta(p,q))=&\conv{\mathrm{int}(2\Delta(p,q))\cap\Z^3}\\
=&\conv{\mathrm{int}(\conv{(0,0),(1,0),(p,q),(p+1,q)})\cap \Z^2}\times \{1\}\\
=&\conv{\left\{\left( \ceil{ \frac{jp}{q} },j,1\right) \mid 0< j< q\right\}}.
\end{align*}
Thus $F(2\Delta(p,q))=\emptyset$ if and only if $q=1$. If the interior lattice points are collinear, we get after a shearing $p=\pm 1$. So we get for $q>1$
\begin{align*}
\dim(F(2\Delta(p,q)))=
\begin{cases}
	0 & \text{ if } q=2\\
	1 & \text{ if } q\geq 3 \text{ and } p\equiv \pm 1 \mod q\\
	2 & \text{ else. }
\end{cases}
\end{align*}

\end{ex}

\section{Special multipliers of a lattice polytope}

In this section we focus on special dilations of a rational $d$-polytope $P$ corresponding to vertices of $\mathcal{F}(P)$.
	
\begin{defi}
Let $P\subseteq M_\R$ be a rational $d$-polytope, $\{(\mu_1,p_1),\dotsc,(\mu_n,p_n)\}$ the set of vertices of $\mathcal{F}(P)\subseteq \R \times M_\R$.\\
We call every $\mu_i$ a \textit{special multipier of $P$}, $\mu(P):= \min_i \mu_i$ the \textit{minimal multiplier}, and $\mu_{max}(P):= \max_i \mu_i$ the \textit{maximal multiplier of $P$}.
\end{defi}

Since $P$ is rational, we now directly get the rationality of all special multipliers.

\begin{coro}\label{rationality_multipliers}
Let $P\subseteq M_\R$ be a rational $d$-polytope and $(-\Min{P}{\nu_k}, \nu_k)\in \Q^{d+1}$ for $1\leq k\leq d+1$  linear independent vectors defining a vertex $(\mu_i, p_i)$ of the polyhedron $\mathcal{F}(P)\subseteq \R \times M_\R$.\\ Then we have
\begin{align*}
\mu_i=\frac{\sum_{j=1}^{d+1}(-1)^{j+d+1}\begin{vmatrix}\nu_1 & \dots &  \hat{\nu_j} & \dots & \nu_{d+1}\end{vmatrix}}{\sum_{j=1}^{d+1}(-1)^{j+d}\Min{P}{\nu_j}\begin{vmatrix}\nu_1 & \dots &  \hat{\nu_j} & \dots & \nu_{d+1}\end{vmatrix}}\in \mathbb{Q},
\end{align*}
were the vectors $\hat{\nu_j}$ are canceled.
\end{coro}
\begin{proof}
This follows directly from Cramer's rule and the Laplace expansion of the occurring determinants.
\end{proof}

In particular, we get now an easy combinatorial proof for the characterization of the minimal multiplier in \cite[3.8.]{Bat23} and for its rationality proven in \cite[3.2.]{Bat24} by methods from toric geometry.
	
\begin{coro}\label{min_mult_rational} Let $P\subseteq M_\R$ be a full dimensional rational polytope.\\
Then $\mu(P)\in \mathbb{Q}$ and $0\leq \dim F(\lambda P)\leq d-1$ if and only if $\lambda=\mu(P)$.
\end{coro}
\begin{proof}
We have $\mu(P)\in \Q$ by \ref{rationality_multipliers}. Moreover, the intersection of $\mathcal{F}(P)\subseteq M_\R\times \R$ with $M_\R \times \lambda$ has dimension between $0$ and $d-1$ if and only if $\lambda=\mu(P)$ because $\mathcal{F}(P)$ has no facet parallel to $M_\R \times \{0\}$ since $\nu \neq 0$ for all $\nu \in \normalfan{P}^\mathrm{can}[1]$.
\end{proof}

For the lattice pyramid $\mathrm{Pyr}(P)$ we can compute the minimal multiplier from the minimal multiplier of $P$ using the results from the last section. 
	
\begin{coro}\label{multipliers_pyramid}
Let $P \subseteq M_\R$ be a lattice polytope with $\mu:=\mu(P)\geq 1$. Then we get $\mu(\mathrm{Pyr}(P))=\mu+1$ for the minimal multiplier of the lattice pyramid and
\begin{align*}
F((\mu+1)\mathrm{Pyr}(P))\cong F(\mu P).
\end{align*}
\end{coro}
\begin{proof}
Since $\mu\geq 1$ we have by \ref{Fine_interior_pyramid}
\begin{align*}
F((\mu+1)\mathrm{Pyr}(P))\cap (\{\lambda\} \times M_\R)=\begin{cases}
F(\mu P) & \text{ if } \lambda=\mu\\
\emptyset & \text{ else.}
\end{cases}
\end{align*}
Thus we get
\begin{align*}
F((\mu+1)\mathrm{Pyr}(P))=F((\mu+1)\mathrm{Pyr}(P))\cap (\{\mu\} \times M_\R ) \cong F(\mu P).
\end{align*}
\end{proof}
	
We now introduce some invariants of Ehrhart theory related to the minimal multiplier. Recall that for a lattice $d$-polytope $P\subseteq M_\R$, by results of Eugène Ehrhart and Richard P. Stanley, we have a unique polynomial $h^*_P \in \Z[x]$ with non-negative coefficients and $\deg h^*_P\leq d$ such that
\begin{align*}
	\sum_{k \in \Z_{\geq 0}} |kP \cap M|x ^k=\frac{h^*_P(x)}{(1-x)^{d+1}}.
\end{align*}
We write $\deg(P) := \deg(h^*_P)$ and have for the \textit{codegree of $P$}, i.e. $\mathrm{codeg}(P):=d+1-\deg(P)$ that $\mathrm{codegree}(P)= \min \{k \in \Z_{>0} \mid |\interior{kP} \cap M|>0\}$ and the highest non-zero coefficient of the $h*$ polynomial gives the number of interior lattice points of $\mathrm{codeg}(P)\cdot  P$. For this and more results on Ehrhart theory see for example \cite{BS15}. We get the following connection with the minimal multiplier.

\begin{coro}	
Let $P\subseteq M_\R$ be a lattice $d$-polytope. Then we have the as a upper bound for the minimal multiplier $\mu(P)\leq \mathrm{codegree}(P)\leq d+1$.
\end{coro}

Since a rational $d$-polytope with lattice width smaller than $2$ has empty Fine interior, we get also a lower bound for $\mu(P)$.
	
\begin{coro} Let $P$ be a rational $d$-polytope. Then $\frac{2}{\lw{P}}\leq \mu(P)$ and if we have $\dim(F(\mu P))=d-1$ then $\frac{2}{\lw{P}}= \mu(P)$.
\end{coro}	
	
We look now at the special situation with $\mu(P)>1$ and $\dim(F(\mu(P)P))=0$ since we want to understand this situation for $d=3$ in details later. First recall some definitions from \cite{Bat24}.

\begin{defi}
Let $P\subseteq M_\R$ be a lattice $d$-polytope. We call $P$ \textit{$F$-hollow} if $\mu(P)>1$. If $\mu(P)>1$ and $\dim(F(\mu(P)P))=0$, then we call $P$ \textit{weakly sporadic $F$-hollow}. A weakly sporadic $F$-hollow lattice $d$-polytope, which is not affine unimodular equivalent to a subset of $P' \times \R^{d-k}$ for some lattice $k$-polytope $P'$ with $k<d$, is called a \textit{sporadic $F$-hollow polytope.} 
\end{defi}

We will no give to classes of examples to illustrate these definitions.

\begin{ex}[\cite{AWW09}]
For any unit fraction partition of $1$ of length $d$, i.e. a $d$-tuple $(k_1,\dotsc,k_d)\in \Z_{\geq 2}^d$ with $\sum_{i=1}^{d}\frac{1}{k_i}=1$, the simplex
\begin{align*}
	\Delta_{(k_1,\dotsc,k_d)}:=\conv{0,k_1e_1,\dotsc, k_de_d}\subseteq \R^d,
\end{align*}
where $e_i$ is the standard basis in $\R^d$, is a sporadic $F$-hollow simplex with
\begin{align*}
	\mu(\Delta_{(k_1,\dotsc,k_d)})=\frac{k+1}{k}
\end{align*}
Moreover, it is a maximal hollow lattice polytope, i.e. there are no hollow lattice polytopes, that contain it as a subpolytope. How do we see this? With $k:=lcm(k_i)$ we can describe the points $x=(x_1,\dotsc,x_d)\in 	\Delta_{(k_1,\dotsc,k_d)}$ as the points in the intersection of the half-spaces $x_i\geq 0$ and $\sum_{i=1}^{d} \frac{k}{k_i}x_i \leq k$. We have no interior lattice points, since every interior point has $x_i>0$ and $\frac{k}{k_i}x_i < k$. Moreover, every facet has at least one lattice point in its relative interior, we can pick the points $\sum_{i\in \{1,\dotsc,d\}, i\neq j} e_i$ and $\sum_{i\in \{1,\dotsc,d\}} e_i$ to see this. So we see that $\Delta_{(k_1,\dotsc,k_d)}$ is a maximal hollow lattice simplex. So it is enough to show that
\begin{align*}
F\left(\frac{k+1}{k} \Delta_{(k_1,\dotsc,k_d)}\right)=\{(1,1,\dotsc,1)\}.
\end{align*}
We get $F\left(\frac{k+1}{k} \Delta_{(k_1,\dotsc,k_d)}\right)\subseteq \{(1,1,\dotsc,1)\}$ since every facet has lattice distance $1$ from $(1,1,\dotsc,1)$. Every other supporting integral hyperplane has also at least lattice distance $1$ to $(1,1,\dotsc,1)$, since there is a vertex of $\Delta_{(k_1,\dotsc,k_d)}$ which has a smaller lattice distance to the hyperplane than $(1,1,\dotsc,1)$ and the vertex is a lattice point.

It is worth noting, that unit fraction partitions of $1$ and generalized forms of them are also used in other classifications of lattice simplices, e.g. in recent classification results of Fano simplices for various Gorenstein indices in \cite{Bae25}.
\end{ex}

\begin{ex}
We can construct some special weakly sporadic lattice polytopes from canonical Fano polytopes. For some canonical Fano polytopes $P$ there is a $\lambda < 1$ and some $x \in \Q^d$ such that $Q:=\lambda P^*-x$ is an $F$-hollow lattice polytope. Then $Q$ is weakly sporadic with $\mu(P)=\mu_{max}(P)=\frac{1}{\lambda}$ by \ref{Fine_lattice_cone}. If $P$ is not only canonical Fano but also reflexive, we get all Gorenstein polytopes and all their integral dilations without interior lattice points this way. In particular, since all canonical Fano polygons are reflexive, we get in dimension $2$ only the three Gorenstein polygons of index greater $1$ and the twofold standard lattice triangle, and we will see later that these are already all weakly sporadic lattice polygons. In dimension $3$ we get $53$ weakly sporadic lattice polytopes this way, among them the $34$ Gorenstein polytopes with index greater than $1$ and the $3$ maximal sporadic simplizes $\Delta_{(3,3,3)}, \Delta_{(2,4,4)}, \Delta_{(2,3,6)}$. If we allow $Q$ to have a non-empty Fine interior but still want it to be hollow, then we also get all maximal hollow lattice polytopes with Fine interior of dimension $0$ and (perhaps surprisingly) also the only maximal hollow lattice polytope with Fine interior of dimension $3$, which are described in \cite{BKS22}.
\end{ex}

We end this section with some words on the other special multipliers.

\begin{lemma}\label{special_multiplier_support}
Let $P$ be a rational $d$-polytope, $n \in S_F(\lambda P)$ for some $\lambda \in \R_{\geq 0}$. Then there is a special multiplier $\mu_i$ of $P$, with $n \in S_F(\lambda P)$ if and only if $\lambda \geq \mu_i$.
\end{lemma}
\begin{proof}
Since $n \in S_F(\lambda P)$ for some $\lambda \in \R_{\geq 0}$, there is a supporting hyperplane of $\mathcal{F}(P)$ with normal vector $(-\Min{P}{n},n)$. The intersection of this hyperplane with $\mathcal{F}(P)$ is a unbounded face of $\mathcal{F}(P)$ and the lowest last coordinate of the vertices of this facet is a special multiplier $\mu_i$ with $\nu \in S_F(\lambda P)$ if and only if $\lambda \geq \mu_i$.
\end{proof}

\begin{defi}
Let $P$ be a rational $d$-polytope, $n \in S_F(\lambda P)$ for some $\lambda \in \R_{\geq 0}$. Then we define $\mu_n$ as the special multiplier from \ref{special_multiplier_support} and call it the \textit{special multiplier of the support vector $n$}. 
\end{defi}

There is a special multiplier connected to the property of being canonically closed introduced in \cite[4.17.]{Bat23}. We also get the characterization and rationality of this multiplier now as a corollary.
	
\begin{coro}
Let $P$ be a rational polytope.\\
There is a multiplier $\mu_{cc}(P)\in \Q$ of $P$, so that $\lambda P$ is canonically closed if and only if $\lambda\geq \mu_{cc}$.
\end{coro}
\begin{proof}
By \cite[4.3]{Bat23} $\lambda P$ is canonically closed if and only if $\Sigma_P[1]\subseteq S_F(\lambda P)$. So by \ref{special_multiplier_support} the multiplier $\mu_{cc}$ can be defined by $\mu_{cc}:=\max \{\mu_n \mid n \in \Sigma_P[1]\}$.
\end{proof}

There are examples with $\mu_{cc}(P)\neq \mu_{max}(P)$ e.g. one can see that
\begin{align*}
	\conv{(-4,-7,-9,-5), (0,1,0,0), (1,0,0,0), (2, 5, 9, 5), (0, 1, 0, 3)}
\end{align*}
is a maximal hollow lattice $4$-simplex with minimal multiplier $1$, in particular it is canonically closed. But one can calculate that $\mu_{max}=\frac{4}{3}$.

Thus, contrary to \cite[5.8]{Bat23}, it is possible that for a canonically closed lattice polytope $P$ the combinatorial type of some $F(\lambda P)$ with $\lambda>1$ is different from the combinatorial type of $F(P)$.
	
Nevertheless, for $\lambda\geq\mu_{max}(P)$ we still have a Minkowski sum decomposition of $F(\lambda P)$ as in \cite[5.5]{Bat23}. But here this result is now just a corollary of \ref{main_theorem} and the usual decomposition of a polyhedron into the Minkowski sum of a polytope and the recession cone.

\begin{coro}\label{FinePolyhedron}
Let $P$ be a rational polytope and $\mu_{max}$ the maximal multiplier of $P$. Then we have
\begin{align*}
F(\lambda \mu_{max}P)=F(\mu_{max}P)+(\lambda-1)\mu_{max}P
\end{align*}
for all $\lambda\geq 1$ and $\mu_{max}$ is the smallest rational number with this property.
\end{coro}
\begin{proof}
We look at $\mathcal{F}(\mu_{max}P)$ for $s_0\geq 1$. If we decompose this polyhedron as Minkowski sum of a polytope and the recession cone, we get since there are only vertices with first coordinate $1$
\begin{align*}
\mathcal{F}(\mu_{max}P) \cap \{x_0\geq 1\}=\{1\} \times F(\mu_{max}P)+\R_{\geq 0}(\{1\} \times \mu_{max}P).
\end{align*}
So we have for all $\lambda\geq 1$
\begin{align*}
\{\lambda\} \times F(\lambda \mu_{max}P)=&\mathcal{F}(\mu_{max}P) \cap \{x_0= \lambda\}\\
=&\{1\} \times F(\mu_{max}P)+(\lambda-1)(\{1\} \times \mu_{max}P)\\
=&\{\lambda\} \times (F(\mu_{max}P)+(\lambda-1)\mu_{max}P).
\end{align*}
For smaller rational numbers $\lambda<\mu_{max}$ the polyhedron $\mathcal{F}(\lambda P) \cap \{y\geq 1\}$ has vertices with first coordinate greater than $1$ and so $\lambda$ can not fulfil the property.
\end{proof}

	
\section{Weakly sporadic $F$-hollow lattice 3-polytopes}
	
Our aim in this section is to classify all weakly sporadic $F$-hollow lattice 3-polytopes.

In a first step we look in arbitrary dimension $d\geq 2$ on the weakly sporadic $F$-hollow lattice $d$-polytopes with degree at most $1$, where the degree is defined as the degree of the $h^*$ polynomial as before. Since lattice $d$-polytopes with degree at most $1$ are classified in \cite{BN07}, we can use this classification to describe all weakly sporadic $F$-hollow lattice $d$-polytopes among them.
	
\begin{prop}\label{weakly_sporadic_small_degree}
Let $P$ be a lattice $d$-polytope with $\deg(P)\leq 1$. Then $P$ is a weakly sporadic $F$-hollow lattice polytope if and only if $P$ is either the standard simplex $\Delta_d$, a Gorenstein polytope of index $d$, or affine unimodular equivalent to the $(d-2)$-fold pyramid over $2\Delta_2$.
\end{prop}
\begin{proof}
If $\deg(P)=0$, then $P$ is by \cite[Prop. 1.4.]{BN07} affine unimodular equvivalent to the standard simplex $\Delta_d$ which is weakly sporadic by \ref{multipliers_pyramid}.
	
So only the case $\deg(P)=1$ remains. By \cite[Theorem 2.5]{BN07} every 3-polytope of degree $1$ is affine unimodular equivalent to the the $(d-2)$-fold pyramid over $2\Delta_2$, which is weakly sporadic $F$-hollow by \ref{multipliers_pyramid}, or to the a Lawrence prism $P_{h_1,\dotsc,h_d}$. Since every Lawrence prism $P_{h_1,\dotsc,h_d}$ projects to the standard simplex $\Delta_{d-1}$, we get $\mu(P_{h_1,\dotsc,h_d})\geq \mu(\Delta_{d-1})=d$. So $\mu(P_{h_1,\dotsc,h_d})=d$, because $dP_{h_1,\dotsc,h_d}$ has interior lattice points since $\mathrm{codeg}(P_{h_1,\dotsc,h_d})=d+1-\deg(P_{h_1,\dotsc,h_d})=d$. To be weakly sporadic, we must therefore have exaktly one interior lattice point in the lattice polytope $d P_{h_1,\dotsc,h_d}$, so we get $h^\ast(P_{h_1,\dotsc,h_d})(t)=t+1$ and so $P_{h_1,\dotsc,h_d}$ is a Gorenstein polytope of index $d$.
\end{proof}

\begin{rem}\label{Gorenstein_d-polytopes}
We can explicitly describe the Gorenstein $d$-polytopes of index $d$ for $d\geq 2$. They all have the $h^\ast$-polynomial $t+1$ and so they are affine unimodular equivalent to a Lawrence prism $P_{h_1,\dotsc,h_d}$ with
\begin{align*}
h^*_{P_{h_1,\dotsc,h_3}}(t)=(h_1+\dotsc+h_d-1)t+1=t+1
\end{align*}
by \cite[2.4. f.]{BN07}. Because of the automorphisms of $\Delta_d$ we get without restriction $(h_1,h_2,\dotsc,h_d)\in \{(1,1,0,\dotsc,0),(2,0,0,\dotsc,0)\}$ and so there are exactly two Gorenstein polytopes of index $d$, namely $P_{1,1,0,\dotsc,0}$ and $P_{2,0,\dotsc,0}$.
\end{rem}

In dimension $2$ every $F$-hollow lattice polygon has degree at most $1$ and so we get already the complete classification of weakly sporadic lattice polygons from this general result, i.e. the weakly sporadic lattice polygons are $\Delta_2, P_{1,1}, P_{2,0}$ and $2 \Delta_2$.

We will see now, that in dimension $3$ the only addional examples of lattice width $1$ are Gorenstein polytopes of index $2$.
	
\begin{prop}
Let $P$ be a lattice $3$-polytope of lattice width 1 with $\deg(P)>1$. Then $P$ is a weakly sporadic $F$-hollow lattice polytope if and only if $P$ is a Gorenstein polytope of index $2$. 
\end{prop}
\begin{proof}
Since $P$ has lattice width $1$, we have $\mu(P)\geq 2$. We can not have $\mu(P)>2$, since then $\deg P=d+1-\mathrm{codeg}(P)\leq 4-3=1$.
		
It remains to look at the case $\mu(P)=2$, and so we have $\dim F(2P)=0$. From \ref{FineInterior_Width1} we also get $\dim F(2P')=0$ for the middle lattice polygon $2P'$, and so $F(2P')$ is a lattice point and $2P'$ is canonically closed. So $F(2P)$ is a lattice point and $2P$ is canonically closed by \ref{FineInterior_Width1}, so $2P$ is a reflexive polytope and therefore $P$ is a Gorenstein polytope of index $2$.
\end{proof}
	
\begin{rem}
The Gorenstein $d$-polytopes of index $d-1$, or equivalently degree $2$, are completely classified in \cite{BJ10}.

In dimension $3$ there are exactly $31$ Gorenstein polytopes of degree $2$, among them the $16$ lattice pyramids over the $16$ reflexive polygons. The polytope $2\Delta_3$ is the only one of these polytopes, which has lattice width greater than $1$.
\end{rem}

We describe now the weakly sporadic $F$-hollow lattice 3-polytopes projecting to $2\Delta_2$ but not to $\Delta_1$ as lattice polytopes in $2\Delta_2\times \R$.
	
\begin{lemma}\label{prism_weakly_sporadic}
Let $P$ be a weakly sporadic $F$-hollow lattice 3-tope.\\
If the lattice width of $P$ is greater than $1$ and $P$ is not a sporadic $F$-hollow polytope, then $P$  is affine unimodular equivalent to a lattice polytope in $2\Delta_2 \times [0,4]$.
\end{lemma}
\begin{proof}
Since $P$ is not sporadic and has a lattice width greater than 1, it must have a lattice projection on $2\Delta_2$ and is therefore affine unimodular equivalent to a lattice polytope $Q$ in $2\Delta_2\times \mathbb{R}$.
		
Let $c=(\frac{2}{3},\frac{2}{3})$ be the barycenter of $2\Delta_2$ and $c_l, c_u$ the lower and upper intersection point of $\{c\}\times \R$ and the boundary of $Q$. The intersection points are points of a lower and an upper facet $F_l$ and $F_u$ of $Q$. The lower facet $F_l$ projects along $(0,0,1)$ onto a lattice subpolygon of $2 \Delta_2\times \{0\}$. There are up to automorphisms of $2\Delta_2\times \{0\}$ two lattice subpolygons without three consecutive vertices which form an affine lattice basis of $\Z^2 \times \{0\}$, namely 
\begin{align*}
\Delta_2 \text{ and } \conv{(0,0,0), (2,0,0), (0,1,0)}.
\end{align*}
So, after a appropriate affine unimodular transformation we can assume that either
\begin{align*}
F_l\subseteq& \Delta_2\times \{0\},\\ F_l=&\conv{(0,0,0),(2,0,0),(0,2,1)} \text{ or }\\ F_l=&\conv{(0,0,1),(2,0,0),(0,1,0)}
\end{align*}
and in particular, we have $Q\subseteq 2\Delta_2 \times [0,\infty)$.
		
Next, we will show that the third coordinate of $c_u$ is at most $\frac{4}{3}$. If $F_l \subseteq \Delta_2\times \{0\}$ or $F_l=\conv{(0,0,1),(2,0,0),(0,1,0)}$, then we have $\frac{3}{2}c_l=(1,1,0)$, and if $F_l=\conv{(0,0,0),(2,0,0),(0,2,1)}$, then we have $\frac{3}{2}c_l=(1,1,\frac{1}{2})$. But since $c_l$ is in the second case an interior point of a facet with integral points of $Q$, we get in both cases that every hyperplane supporting $\frac{3}{2}Q$ from below has at least lattice distance $1$ to $(1,1,1)$. Similarly every hyperplane supporting $\frac{3}{2}Q$ from above has at least lattice distance $1$ to $\frac{3}{2}c_u-(0,0,1)$. If we assume that the third coordinate of $c_u$ is larger than $\frac{4}{3}$, we get $F(\frac{3}{2}Q)\supseteq \conv{(1,1,1),\frac{3}{2}c_u-(0,0,1)}$ and $\dim(F(\frac{3}{2}Q))=1$, which contradicts the assumption that $Q$ is weakly sporadic.
		
Let $h$ be a vertex of $Q$ with a maximum third coordinate. It remains to show, that this third coordinate is at most $4$. The line through $h$ and $c_u$ has two intersection points $h$ and $s$ with the boundary of $\Delta_2 \times \R$ and the third coordinate of $s$ is at least $0$, because between $c_u$ and $s$ the intersection of the line with the interior of $Q$ is empty and the segment between $c_u$ and $s$ lies above the polytope $Q$. Since the length of the segment between $h$ and $c_u$ is at most twice the length of the segment between $c_u$ and $s$, we get that the third coordinate of $h$ is at most $3\cdot \frac{4}{3}=4$. 

Therefore, $P$  is affine unimodular equivalent to a subpolytope $Q$ of $2\Delta_2 \times [0,4]$. 
\end{proof}
	
\begin{rem}
The polytope $2P_{2,0,0}=\conv{(0,0,0),(2,0,0),(0,2,0),(0,0,4)}$ is weakly sporadic $F$-hollow because $P_{2,0,0}$ is a Gorenstein polytope of index $3$ by \ref{Gorenstein_d-polytopes} and has only the lattice width directions $\pm (1,0,0),\pm (0,1,0),\pm(1,1,0)$. This shows that the prism $2\Delta_2\times [0,4]$ in Lemma~\ref{prism_weakly_sporadic} is the best possible choice.
\end{rem}

We can explicitly classify the lattice subpolytopes of $2\Delta_2 \times [0,4]$ by recursively deleting vertices from the lattice polytope and looking at the convex hull of the remaining lattice points. To do this efficiently, we should skip all the lattice polytopes we have already seen interms of affine unimodular equivalence. This can be done, for example, by using the PALP normal form of a lattice polytope (\cite{KS04}, \cite{GK13}). We get the following result.

\begin{prop}
There are exactly $80$ weakly sporadic non sporadic $F$-hollow lattice 3-polytopes of lattice width $2$. All of them are lattice subpolytopes of $2P_{2,0,0}$ or $2P_{1,1,0}$. $2\Delta_3$ is the only one of them with minimal multiplier $\mu=2$, the others have $\mu=\frac{3}{2}$.
\end{prop}
		
\begin{thm}\label{weakly_sporadic classification}
There are exactly $114$ weakly sporadic but non sporadic $F$-hollow lattice $3$-polytopes. Among them are exactly $31$ Gorenstein polytopes of index $2$, $2$ Gorenstein polytopes of index $3$, $1$ Gorenstein polytope of index $4$, $1$ exceptional simplex with $\mu=\frac{5}{2}$ and $79$ polytopes of width $2$ with $\mu=\frac{3}{2}$. Coordinates for the vertices of all these polytopes are available on \cite{Boh24}.
\end{thm}
	
By \cite[Theorem 1.11.]{Bat24}	up to affine unimodular equivalence there are  only finitely many sporadic $F-$hollow lattice $d$-polytopes, i.e. $F$-hollow lattice polytopes without $F$-hollow projection. A similar result holds for hollow lattice polytopes \cite{NZ11}. In dimension $3$ all sporadic $F$-hollow polytopes are subpolytopes of the $12$ maximal hollow lattice polytopes classified in \cite{AWW11}, \cite{AKW17}.
	
Deciding whether a lattice polytope projects to $[0,1]$ is easy, since we have such a projection if and only if the polytope has lattice width $1$. Since the only $F$-hollow polygon with lattice with greater than $1$ is $2\Delta_2=\conv{(0,0), (2,0), (0,2)}$, we need the following lemma to understand projections on this polygon.
		
\begin{lemma}
Let $P$ be a lattice $d$-polytope, $d\geq 3$, with lattice width greater than $1$.\\ Then $P$ projects to $2\cdot \Delta_2$ if and only if $P$ has lattice width $2$ and there are linearly dependent lattice width directions $w_1, w_2, w_3$ of $P$ such that $0$ is a interior point of $Q:=\conv{w_1,w_2,w_3}$, the normalized area of $Q$ is $3$ and the supporting hyperplanes of $P$ corresponding to the lattice width directions $\pm w_1, \pm w_2, \pm w_3$ define a polyhedron with exactly $3$ facets.
\end{lemma}
\begin{proof}
If $P$ projects to $2\Delta_2$, then $P$ is affine unimodular equivalent to a subpolytope of $2\Delta^2 \times \R^{d-1}$. Since $2\Delta_2$ has lattice width $2$ and width directions $(1,0), (0,1), (-1,-1)$, we have lattice width directions $w_1=(1,0,0,\dotsc,0), w_2=(0,1,0,\dotsc,0), w_3=(-1,-1,0,\dotsc,0)$ of $P$ and $w_1,w_2,w_3$ satisfy the required conditions.

Conversely, if we have lattice width directions $w_1,w_2,w_3$ that satisfy the conditions, then $\conv{0, w_1, w_2}$ must be an empty triangle, since $0$ is an interior point of the polygon $\conv{w_1,w_2,w_3}$ with normalized area $3$. Thus $0,w_1,w_2$ is an affine lattice basis of a $2$-dimensional sublattice, which we can extend to a lattice basis of the whole lattice $\Z^d$. Mapping this lattice basis to the standard basis we get that $P$ is affine unimodular equivalent to a subpolytope of $[0,2]^2\times \R^{d-2}$. Under this map, the width direction $w_3=-w_1-w_2$ maps to $(-1,-1,0,\dotsc, 0)$ and so $P$ is even affine unimodular equivalent to a subpolytope of $2\Delta_2\times \R^{d-2}$ or $\conv{(1,0),(0,1),(-1,1),(-1,0),(0,-1),(1,-1)}\times \R^{d-2}$. But the latter is not possible, since the supporting hyperplanes of $P$, corresponding to the lattice width directions, define a polyhedron with exactly $3$ facets, and so $P$ is affine unimodular equivalent to a subpolytope of $2\Delta_2\times \R^{d-2}$ and thus projects to $2\Delta_2$.
\end{proof}

\begin{rem}
Working with the lattice width has been helpful for various classifications of lattice polytopes e.g. for the classification of empty $4$-simplices in \cite{IVS21}. There are also classifications looking at a multi-width, which is even a generalization of our situation with lattice width $2$ for three different width directions, see \cite{Ham24} for this approach.
\end{rem}

With the help of the lemma, we can now determine all sporadic $F$-hollow lattice $3$-polytopes as lattice polygons of the maximal hollow ones. We get the following result. 
	
\begin{thm}\label{sporadic classification}
There are exactly $1368$ sporadic $F$-hollow lattice $3$-polytopes up to affine unimodular equivalence. All of them are subpolygons of $\Delta_{(3,3,3)}$, $\Delta_{(2,4,4)}$ or $\Delta_{(2,3,6)}$. $300$ have $\mu=\frac{4}{3}$, $632$ have $\mu=\frac{5}{4}$ and $436$ have $\mu=\frac{7}{6}$. Coordinates for the vertices of all these polytopes are available on \cite{Boh24}.
\end{thm}

Let us end with two remarks on other classifications.

\begin{rem}
Since by \cite[Appendix B]{BKS22} there are $9$ hollow lattice $3$-polytopes, which are not $F$-hollow, we have all in all $1377$ sporadic hollow lattice $3$-polytopes, not projecting to a hollow lattice polytope of smaller dimension.
\end{rem}

\begin{rem}
Among the sporadic $F$-hollow polytopes are $52$ bipyramids, which were also classified in \cite[Lemma 4.2.]{IVS21} to determine all empty $4$-simplices, which have a hollow projection to a hollow 3-polytope. $29$ primitive bipyramids out of the $52$ correspond to the $29$ \textit{stable quintuples} in the classification of $4$-dimensional terminal quotient singularities in \cite{MMM88}. 
\end{rem}

\vspace{15mm}

\textbf{Acknowledgements.} I would like to thank Victor Batyrev for introducing me to the Fine interior and for his suggestion to study the Fine interior combinatorially for dilations and to classify weakly sporadic $F$-hollow $3$-polytopes.

\end{document}